\documentclass[12pt,twoside,final,psamsfonts]{amsart}

\usepackage[psamsfonts]{amssymb}
\usepackage{times,a4wide}
\usepackage{MnSymbol}

\theoremstyle{plain}

\numberwithin{equation}{section}
\theoremstyle{plain}

\newtheorem{corollary}[equation]{Corollary}
\newtheorem{theorem}[equation]{Theorem}
\newtheorem{lemma}[equation]{Lemma}
\theoremstyle{definition}

\newtheorem{example}[equation]{Example}
\newtheorem{remark}[equation]{Remark}

\newcommand{\nc}{\newcommand}

\newcommand{\D}{{\mathbb D}}

\newcommand{\T}{{\mathbb T}} 

\newcommand{\ADC}{\operatorname{ADC}}

\newcommand{\vp}{\varphi}
\nc{\bea}{\begin{eqnarray}}
\nc{\eea}{\end{eqnarray}}
\nc{\beqa}{\begin{eqnarray*}}
\nc{\eeqa}{\end{eqnarray*}}
\nc{\Hi}{H^{\infty}}
\nc{\loi}{\ell^{\infty}}
\nc{\NL}{N^+\vert \Lambda}
\nc{\hf}{{\mathcal H}_{\phi}}
\nc{\liL}{\lambda\in\Lambda}
\nc{\nn}{\nonumber}
\nc{\dst}{\displaystyle}

\newenvironment{proof*}{\vskip 2mm\noindent {}}{$\blacksquare$ \vskip 2mm}
\numberwithin{equation}{section}

\renewcommand{\Re}{\mbox{Re}}

\title{Truncated Toeplitz operators and boundary values in 
nearly invariant subspaces}

\author{Andreas Hartmann \& William T.\ Ross}

\address{Institut de Math\'ematiques de Bordeaux,
Universit\'e Bordeaux I, 351 cours de la Lib\'eration,
33405 Talence, France}

\address{Department of Mathematics, University of Richmond, VA 23173, USA}

\thanks{This work has been done while the first named author was staying at
University of Richmond as the Gaines chair in mathematics. He would like to thank that institution for the hospitality
and support during his stay. This
author is also partially supported by ANR FRAB}

\email{hartmann@math.u-bordeaux.fr, wross@richmond.edu}

\date{\today}



\begin{document}

\bibliographystyle{amsalpha}

\begin{abstract} We consider truncated Toeplitz operator on nearly invariant
subspaces of the Hardy space $H^2$. Of some importance in this context is the boundary behavior of the
functions in these spaces which we will discuss in some detail.
\end{abstract}

\maketitle

\section{Introduction}

If $H^2$ is the classical Hardy space \cite{Duren}, we say a (closed) subspace $M \subset H^2$ is \emph{nearly invariant} when 
$$f \in M, f(0) = 0 \Rightarrow \frac{f}{z} \in M.$$ These subspaces
have been completely characterized in \cite{Hi88, Sa88} and continued
to be studied in \cite{AR96, HSS04, MP05, KN06, AK08, CCP}. 
In this paper, we will examine certain properties of nearly invariant subspaces. 
We aim to accomplish three things.

 First, in Theorem \ref{prop1.2} we observe that much of what we already know about truncated Toeplitz operators on model spaces \cite{Sa07} can be transferred, via a unitary operator to be introduced
below, \emph{mutatis mutandis} to truncated operators on nearly invariant subspaces. 

Secondly, in Theorem \ref{MNTL} we will show that every function in a nearly invariant subspace $M$ has a finite non-tangential limit at $|\zeta| = 1$ if and only if (i) $g$ has a finite non-tangential limit at $\zeta$ and (ii) the reproducing kernel functions for $M$ are uniformly norm bounded in Stolz regions with vertex at $\zeta$. This parallels a result by
Ahern and Clark in model spaces.

Third, in order to better understand the self-adjoint rank-one truncated Toeplitz operators
on nearly invariant subspaces, we will discuss the non-tangential limits of functions 
in these spaces. 
It turns out (Theorem \ref{thmdicho}) 
there 
is a kind of dichotomy: 
If every function in a nearly invariant subspace $M$ has a boundary limit at a fixed point $|\zeta| = 1$
then either $\zeta$ is a point where {\it every} function in $\frac{1}{g} M$ (where $g$ is the extremal function for $M$ - see definition below) has a finite non-tangential limit or {\it every} function in $M$ has non-tangential limit 0 at $\zeta$.

A word concerning numbering in this paper: in each section, 
we have numbered theorems,
propositions, lemmas, corollaries {\bf and} equations consecutively.


\section{Preliminaries}

If $H^2$ is the Hardy space of the open unit disk $\D$ (with the usual norm $\|\cdot\|$) and $I$ is inner, let $K_I = H^2 \ominus I H^2$ be the well-studied model space \cite{Niktr}. Note that $H^2$, as well as $K_{I}$, are regarded, via non-tangential boundary values on $\T := \partial \D$, as subspaces of $L^2 := L^2(\T, d\theta/2\pi)$ \cite{Duren}.  It is well known that $H^2$ is a reproducing kernel Hilbert space with kernel 
$$k_{\lambda}(z) := \frac{1}{1 - \overline{\lambda} z},$$ as is $K_I$ with kernel
$$k^{I}_{\lambda}(z) := \frac{1 - \overline{I(\lambda)} I(z)}{1 - \overline{\lambda} z}.$$ Note that
$k_{\lambda}^{I}$ is bounded and that finite linear combinations of 
them form a dense subset of $K_I$. Also note that if $P_I$ is the orthogonal projection of $L^2$ onto $K_I$, then 
$$(P_{I} f)(\lambda) = \langle f, k_{\lambda}^{I} \rangle, \quad \lambda \in \D.$$
Note that for every $\lambda\in\D$ this formula extends to $f\in L^1
=L^1(\T,d\theta/2\pi)$ so that we can define $P_I$ for
$L^1$-functions (which in general does not give an $H^1$-function).

By a theorem of Ahern and Clark \cite{AC70}, every function $f \in K_I$ has a finite non-tangential limit at a boundary point $\zeta \in \T$, i.e., 
$$f(\zeta) := \angle \lim_{\lambda \to \zeta} f(\lambda)$$
exists for every  $f \in K_I$,
if and only if $\zeta \in \ADC(I)$. Here $\zeta \in \ADC(I)$ means that 
$I$ has a finite angular derivative in the sense of Carath\'{e}odory, meaning
$$\angle \lim_{\lambda \to \zeta} I(\lambda)  = \eta \in \T$$ and 
$$\angle \lim_{\lambda \to \zeta} I'(\lambda) \mbox{\; \; exists}.$$
Moreover, whenever $\zeta \in \ADC(I)$, the linear functional $f \mapsto f(\zeta)$ is continuous on $K_I$ giving us a kernel function $k_{\zeta}^{I}$ for $K_I$ at the boundary point $\zeta$. That is to say
$$f(\zeta) = \langle f, k_{\zeta}^{I} \rangle \quad \forall  f \in K_I.$$

Sarason \cite{Sa07} began a study, taken up by others, of the \emph{truncated Toeplitz operators} on $K_I$. These operators are defined as follows: For $\vp \in L^2$, define the truncated Toeplitz operator $A_{\vp}$ densely on the bounded $f \in K_I$ by $$A_{\vp} f  = P_I (\vp f).$$ Let $\mathcal{T}_I$ denote the $A_{\vp}$ which extend to be bounded operators on $K_I$. Certainly when $\vp$ is bounded, then $A_{\vp} \in \mathcal{T}_{I}$, but there are unbounded $\vp$ which yield $A_{\vp} \in \mathcal{T}_{I}$. Moreover, there are bounded truncated Toeplitz operators which can not be represented by a bounded symbol \cite{BCFMT}. 
Much is known about these operators (see the Sarason paper \cite{Sa07} for a detailed discussion) but we list a few interesting facts below:
\begin{enumerate}
\item $\mathcal{T}_{I}$ is a weakly closed linear subspace of operators on $K_I$. 
\item $A_{\vp} \equiv 0$ if and only if $ \vp \in I H^2 + \overline{I H^2}$. 
\item The operator $C f := \overline{z f} I$ (considered as boundary functions on $\T$) defines an isometric, anti-linear, involution on $K_I$ for which $C A_{\vp} C = A_{\vp}^{*} = A_{\overline{\vp}}$. The operator $C$ is called a conjugation. This makes $\mathcal{T}_{I}$ a collection of complex symmetric operators \cite{GP06, GP07}. 
\item A bounded operator $A$ on $K_I$ belongs to $\mathcal{T}_I$ if and only if there are functions $\vp_1, \vp_2 \in K_I$ so that 
$$A - A_z A A_{z}^{*} = (\vp_1 \otimes k^{I}_{0}) + (k^{I}_{0} \otimes \vp_2).$$ 
In the above, we use the notation $f \otimes g$ for the rank-one operator 
$$(f \otimes g) (h) := \langle h, g \rangle f.$$
\item An operator $A$ on $K_I$ is a 
rank-one truncated Toeplitz operator on if and only if it can be written as constant multiple of one of the
following three types 
\[
 k_{\lambda}^I\otimes Ck_{\lambda}^I,
 \quad C k_{\lambda}^I\otimes k_{\lambda}^I,
 \quad k_{\zeta}^I\otimes k_{\zeta}^I
\]
where $\lambda\in\D$, $\zeta\in \ADC(I)$. The two first ones are the non-selfadjoint
and the last ones are the selfadjoint truncated Toeplitz operators.
\item Sedlock \cite{Sed} (see also \cite{GRW})  showed that every maximal algebra in $\mathcal{T}_{I}$ can be written as the commutant of a generalization of the Clark unitary operator. 
\item In \cite{TTOSIUES} they show, for two inner functions $I_1$ and $I_2$, that $\mathcal{T}_{I_1}$ is spatially isomorphic to $\mathcal{T}_{I_2}$ if and only if either $I_1 = \psi \circ I_2 \circ \phi$ or $I_1 = \psi \circ \overline{I_2(\overline{z})} \circ \phi$ for some disk automorphisms $\phi, \psi$. 
\end{enumerate}       

As we have already mentioned,
a \emph{nearly invariant} subspace $M$ is a closed subspace of 
$H^2$ such that 
$$f \in M, f(0) = 0 \Rightarrow \frac{f}{z} \in M.$$
A result of Hitt \cite{Hi88} says that if $g$ (called the extremal function for $M$) is the unique solution to the extremal problem
\begin{equation} \label{extg}
 \sup\{\Re g(0):g\in M, \|g\| =1\},
\end{equation}
then there is an inner function $I$ so that 
$$M = g K_I$$ and moreover, the map 
$$U_{g}: K_I \to M, \quad U_{g} f = g f,$$ is isometric. 

 Note that necessarily we have the condition $I(0)=0$ since $g\in
 M=gK_I$ and so $1\in K_I$. If one were to choose any $g$ and any
 inner $I$, then $g K_I$ is, in general, not even a subspace of $H^2$. 
Sarason shows in
 \cite{Sa88} that when $I(0) = 0$, every isometric multiplier on $K_I$ takes the form
\begin{equation} \label{Sar-g}
 g=\frac{a}{1-Ib},
\end{equation}
where $a$ and $b$ are in the unit ball of $\Hi$ with $|a|^2+|b|^2=1$ a.e.\ on $\T$. As a consequence, $g K_I$ is a (closed) nearly invariant subspace of $H^2$ with extremal function $g$ as in \eqref{extg}.

\begin{remark}
From now on, whenever we speak of a nearly invariant subspace $M = g K_I$, we will always assume that $I(0) = 0$ and that $g$ is extremal for $M$ (as in \eqref{extg}) and of the form in \eqref{Sar-g}. We will say that $g K_I$ is a nearly invariant subspace with extremal function $g$ and suitable inner function $I$ with $I(0) =0$.
\end{remark}

Our first step towards defining a truncated Toeplitz operator on the nearly invariant subspace $M = g K_I$, as Sarason did for $K_I$, is to understand $P_M$, the orthogonal projection of $L^2$ onto $M$. 

\begin{lemma}\label{Lemma1.1}
Let $M=gK_I$ be a nearly invariant subspace with extremal function $g$ and associated
inner function $I$, $I(0)=0$. Then
\[
 P_M f=gP_I(\overline{g} f), \quad f \in M.
\]
\end{lemma}

\begin{proof}
To show that the map $f \mapsto gP_I(\overline{g} f)$ is indeed $P_M$, we need to show that this map is the identity on $M$ and vanishes on $L^2 \ominus M$. 

Let $f = g h \in g K_I$ with $h \in K_I$. Then for every $\lambda \in \D$ we have 
$$(P_I \overline{g} g h)(\lambda) = \langle \overline{g} g h, k^{I}_{\lambda} \rangle = \langle g h, g k^{I}_{\lambda}\rangle.$$ But, from our previous discussion,  multiplication by the extremal function $g$ is an isometry from $K_I$ onto $M$ and so
$$\langle g h, g k^{I}_{\lambda} \rangle = \langle h, k^{I}_{\lambda}\rangle  = h(\lambda).$$ Thus 
$$g(\lambda) (P_I \overline{g} f)(\lambda) = f(\lambda),$$ 
in other words, $f \mapsto gP_I(\overline{g} f)$ is the identity on $M$. 

To finish the proof, we need to show that the map $f \mapsto gP_I (\overline{g} f)$ vanishes on $L^2 \ominus M$. 
Clearly, when $f\in \overline{H^2_0} = \{\overline{h}: h \in H^2, h(0) = 0\}$ then $P_I(\overline{g}f)=0$.
When $f \in H^2 \ominus M$ we have 
$$(P_I \overline{g} f)(\lambda) = \langle \overline{g} f, k^{I}_{\lambda} \rangle = \langle f, g k^{I}_{\lambda} \rangle = 0$$
since $g k^{I}_{\lambda} \in M$ and $f \perp M$. 
\end{proof}

\begin{corollary}
The reproducing kernel for $M = g K_{I}$ is given by 
$$k_{\lambda}^{M}(z) = \overline{g(\lambda)} g(z) \frac{1 - \overline{I(\lambda)} I(z)}{1 - \overline{\lambda} z}.$$
\end{corollary}

\begin{proof}
Observe that $k_{\lambda}^M(z)=(P_Mk_{\lambda})(z)=g(z)(P_I \overline{g}k_{\lambda})(z)$ and 
\[
 (P_I\overline{g}k_{\lambda})(z)=\langle \overline{g}k_{\lambda},k_z^I\rangle
 =\langle k_{\lambda},gk_z^I\rangle=\overline{(gk_z^I)(\lambda)}=\overline{g(\lambda)}
 k_{\lambda}^I(z).
\]
\end{proof}

\section{Truncated Toeplitz operators on nearly invariant subspaces}

We are now  able to introduce truncated Toeplitz operators on nearly
invariant subspaces. Certainly whenever $\vp$ is a bounded function we can use Lemma \ref{Lemma1.1}
to see  that the operator
\[
 A^M_{\vp} f: =P_M\vp f=gP_I(\overline{g}\vp f), \quad f \in M,
\]
is well-defined and bounded. 

The most general situation is when the symbol $\vp$ is a Lebesgue measurable function on $\T$ and $|g|^2\vp\in L^2$. Then, for every bounded $h  \in K_I$,
the function $|g|^2\vp h$ belongs to $L^2$ and so $P_I(|g|^2\vp h)\in K_I$, and, by
the isometric multiplier property of $g$ in $K_I$, we get 
$$P_M(\vp h)  = g P_I(|g|^2\vp h)  \in gK_I= M.$$
Note that by the isometric property of $g$,
the set $gK^{\infty}_I$, where $K^{\infty}_I$ is the set of bounded functions in $K_I$,
is dense in $gK_I$. Hence, in this situation, $A^M_{\vp}$ is densely defined.

Let $\mathcal{T}^{M}$ denote the densely defined $A_{\vp}^{M}$, $|g|^2 \vp \in L^2$, which extend to be bounded on $M$ and recall that $U_{g}: K_I \to M, U_g f = g f$, is  unitary.

\begin{theorem}\label{prop1.2}
Let $M=gK_I$ be a nearly invariant subspace with extremal function $g$ and suitable
inner function $I$, $I(0)=0$. Then for any Lebesgue measurable function $\vp$ on $\T$ with $|g|^2 \vp \in L^2$ we have 
$$U_{g}^{*} A^{M}_{\vp} U_{g} = A_{|g|^2 \vp}.$$
\end{theorem}

\begin{proof}
Let $h\in K_I$, then
\[
 A^M_{\vp}U_gh=A^M_{\vp}gh=gP_I|g|^2\vp h=U_gA_{|g|^2\vp}h.
\]
\end{proof}

The above shows that the map 
$$A_{\vp}^{M} \mapsto A_{|g|^2 \vp}$$ establishes a spatial isomorphism between $\mathcal{T}^{M}$ (the bounded truncated Toeplitz operators on $M$) and $\mathcal{T}_{I}$ (the bounded truncated Toeplitz operators on $K_I$) induced, via conjugation, by the unitary operator $U_{g}: K_I \to M = g K_I$, i.e., $\mathcal{T}^{M} = U_{g} \mathcal{T}_{I} U_{g}^{*}$.

In view of the results in Sarason's paper \cite{Sa07} 
one can use the above spatial isomorphism to prove the following facts: 
\begin{enumerate}
\item $\mathcal{T}_{M}$ is a weakly closed linear subspace of operators on $M$. 
\item $A^{M}_{\vp} \equiv 0$ if and only if $|g|^2 \vp \in I H^2 + \overline{I H^2}$. 
\item Recalling the conjugation $C$ mentioned earlier, define $C_{g} := U_g C U_{g}^{*}$. This defines a conjugation on $M$ and $C_{g} A C_{g} = A^{*}$ for every $A \in \mathcal{T}^{M}$,
and since $(A^M_{\vp})^*=A^M_{\overline{\vp}}$ (as can be verified by direct inspection),
we see that $\mathcal{T}^{M}$ is also a collection of complex symmetric operators
(with respect to $C_g$).
\item If $S_{g} := U_g A_z U_{g}^{*}$, then a bounded operator $A$ on $M$ belongs to $\mathcal{T}^{M}$ if and only if there are functions $\vp_1, \vp_2 \in M$ so that 
$$A - S_{g} A S_{g}^{*} = (\vp_1 \otimes k^{M}_{0}) + (k^{M}_{0} \otimes \vp_2).$$ 
\item Recall from our earlier discussion of \cite{Sa07} that $A$ is a 
rank-one truncated Toeplitz operator on $K_I$ if and only if it can be written as constant multiple of one of the
following three types (two of which are non-selfadjoint, and one of which is
selfadjoint)
\[
 k_{\lambda}^I\otimes Ck_{\lambda}^I,
 \quad C k_{\lambda}^I\otimes k_{\lambda}^I,
 \quad k_{\zeta}^I\otimes k_{\zeta}^I
\]
where $\lambda\in\D$, $\zeta\in \ADC(I)$. It follows immediately that for $\lambda\in\D$,
\beqa
 U_g(k_{\lambda}^I\otimes Ck_{\lambda}^I)U_g^*=U_gk_{\lambda}^I\otimes U_gCk_{\lambda}^I
 =gk_{\lambda}^I\otimes gCk_{\lambda}^I,\\
 U_g(Ck_{\lambda}^I\otimes k_{\lambda}^I)U_g^*=U_gCk_{\lambda}^I\otimes U_gk_{\lambda}^I
 =gCk_{\lambda}^I\otimes gk_{\lambda}^I,
\eeqa
are rank-one truncated Toeplitz operators on $M = gK_I$. What are the self-adjoint truncated Toeplitz operators on $M$? The result here is the following: 
Suppose that $\zeta\in \ADC(I)$, then
\[
 gk_{\zeta}^I\otimes gk_{\zeta}^I
\]
is a self-adjoint rank-one truncated Toeplitz operator on $M$.
Conversely, if $A$ is a non-trivial rank-one self adjoint truncated Toeplitz operator on $M$, then there exists a $\zeta \in \ADC(I)$ so that $A= g k_{\zeta}^{I} \otimes g k_{\zeta}^{I}$. 
\item Recall from our earlier discussion that every maximal algebra in $\mathcal{T}_{I}$ can be written as the commutant of a generalization of the Clark unitary operator. Conjugating by $U_g$ we get an analogous result for the maximal algebras of $\mathcal{T}^{M}$. 
\item Again from our earlier discussion, for two inner functions $I_1$ and $I_2$, that $\mathcal{T}_{I_1}$ is spatially isomorphic to $\mathcal{T}_{I_2}$ if and only if either $I_1 = \psi \circ I_2 \circ \phi$ or $I_1 = \psi \circ \overline{I_2(\overline{z})} \circ \phi$ for some disk automorphisms $\phi, \psi$. Using Theorem \ref{prop1.2} we get the exact same result for $\mathcal{T}^{M_1}, \mathcal{T}^{M_{2}}$. 
Notice how this result is independent of the corresponding
extremal functions $g_1$ and $g_2$.
\end{enumerate}

\section{Existence of non-tangential limits}
With trivial examples one can show that there is no point $\zeta \in \T$ such that every $f \in H^2$ has a non-tangential limit at $\zeta$. However, there are model spaces $K_I$ and $\zeta \in \T$ so that every $f \in K_I$ has a non-tangential limit at $\zeta$. This situation was thoroughly discussed by Ahern and Clark \cite{AC70} with the following theorem. For $\zeta \in \T$ and $\alpha > 1$ let 
$$\Gamma_{\alpha}(\zeta) := \left\{z \in \D: \frac{|z - \zeta|}{1 - |z|} < \alpha \right\}$$ be the standard Stolz domains with vertex at $\zeta$.

\begin{theorem}[Ahern-Clark]\label{thmACbis}
Let $I$ be an inner function. Every function $f\in K_I$ has a non-tangential boundary limit at $\zeta\in\T$
if and only if for every fixed Stolz domain $\Gamma_{\alpha}(\zeta)$ at $\zeta$, the family 
$$F_{\alpha, \zeta}:= \left\{k_{\lambda}^I:\lambda\in \Gamma_{\alpha}(\zeta)\right\}$$ is uniformly norm
bounded, i.e., 
$$\sup\left\{\|k_{\lambda}^{I}\|^2 = \frac{1 - |I(\lambda)|^2}{1 -
    |\lambda|^2}: \lambda \in \Gamma_{\alpha}(\zeta)\right\} 
< \infty.$$
\end{theorem}

The original proof of this theorem involved a technical lemma using operator theory. 
We will give a new proof which avoids this technical lemma. This will allow us to 
consider the situation of nearly invariant subspaces where the operator theory lemma of Ahern-Clark 
no longer works. Also note that the uniform boundedness of the families $F_{\alpha, \zeta}$ is equivalent to the condition $\zeta \in \ADC(I)$. 

\begin{proof}[Proof of Theorem \ref{thmACbis}]
From the uniform boundedness principle it is clear that if every function in $K_I$
has a non-tangential limit at $\zeta\in\T$, then the kernels $k_{\lambda}^I$
are uniformly bounded in Stolz domains $\Gamma_{\alpha}(\zeta)$.

Let us consider the sufficiency part. If, for fixed $\alpha$ and $\zeta$,  the family $F_{\alpha, \zeta}$ is uniformly bounded, then there
exists a sequence $\Lambda := \{\lambda_n\}_{n \geq 1}\subset \Gamma_{\alpha}(\zeta)$ such that $k^{I}_{\lambda_n}$ converges weakly to some $k^{I, \Lambda}_{\zeta} \in K_I$, i.e., for every function $f\in K_I$,
\begin{equation} \label{hold}
 f(\lambda_n)=\langle f,k^{I}_{\lambda_n}\rangle\to \langle f,k_{\zeta}^{I, \Lambda}\rangle
 =:f_{\zeta}^{\Lambda} \quad \mbox{as $n \to \infty$}.
\end{equation}

Observe that 
$$(A_{\overline{z}} f)(\lambda) = \frac{f(\lambda) - f(0)}{\lambda}$$ and 
$$(A_{\overline{z}}^{N} f)(\lambda) = \frac{f(\lambda)-\sum_{j=0}^{N-1}
 \widehat{f}(j)\lambda^j }{\lambda^N},$$ where $\widehat{f}(j)$ is the $j$-th Fourier coefficient of $f$. As a result we have
 $$\|A_{\overline{z}}^{N} f\|^2 = \sum_{j = N}^{\infty} |\widehat{f}(j)|^2 \to 0 \quad \mbox{as $N \to \infty$}.$$

Apply \eqref{hold} to the functions $A_{\overline{z}}^{N} f$ to get
$$
 \langle A_{\overline{z}}^{N} f,k_{\lambda}^I\rangle= 
 (A_{\overline{z}}^{N} f)(\lambda)=\frac{f(\lambda)-\sum_{j=0}^{N-1} \widehat{f}(j) \lambda^j }{\lambda^N}.
$$
Hence
\beqa
 \langle A_{\overline{z}}^N f,k_{\zeta}^{I, \Lambda}\rangle&=&
 \lim_{n\to\infty}\langle A_{\overline{z}}^N f,k_{\lambda_n}^I\rangle
 =\lim_{n\to\infty}\frac{f(\lambda_n)-\sum_{j=0}^{N-1}
 \widehat{f}(j) \lambda_n^j}{\lambda_n^N}\\
 &=&\frac{f_{\zeta}^{\Lambda}-\sum_{j=0}^{N-1}
 \widehat{f}(j)  \zeta^j}{\zeta^N}.
\eeqa
Now using the fact that $\|A_{\overline{z}}^Nf\| \to  0$ when $N \to \infty$, we obtain from
  $\langle A_{\overline{z}}^{N} f, k_{\zeta}^{I, \Lambda}\rangle \to 0$ that
\[
 \lim_{N\to\infty} \frac{f_{\zeta}^{\Lambda}-\sum_{j=0}^{N-1}
 \zeta^j \widehat{f}(j)}{\zeta^N}=0,
\]
and, since the denominator in the above limit is harmless, we get
\[
\sum_{j=0}^{\infty}
  \widehat{f}(j) \zeta^j =f_{\zeta}^{\Lambda}.
\]
This means that the Fourier series for $f$ converges at $\zeta$.

Now take any sequence $\Lambda := \{\lambda_n\}_{n \geq 1}$ converging non-tangentially
to $\zeta$. In view of the uniform boundedness of the family 
$F_{\alpha, \zeta}$, there exists a weakly 
convergent subsequence $\Lambda':= \{k^{I}_{\lambda_{n_l}}\}_{l \geq 1}$.
Repeating the above proof, we obtain 
\[
\sum_{j=0}^{\infty}
 \zeta^j \widehat{f}(j) =f_{\zeta}^{\Lambda'}.
\]
In other words, for every sequence $\{\lambda_n\}_{n \geq 1}$ converging non-tangentially
to $\zeta$ there exists a subsequence $\{\lambda_{n_j}\}_{j \geq 1}$ such that
$\{f(\lambda_{n_j})\}_{j \geq 1}$ converges to the same limit
\[
 l_{\zeta}:=\sum_{j=0}^{\infty}
 \widehat{f}(j) \zeta^j.
\]
It follows that $f$ has a finite non-tangential limit at $\zeta$.
\end{proof}

We will now prove an analog of Theorem \ref{thmACbis} for nearly invariant subspaces where the situation is a bit more complicated. 
Certainly a necessary condition for the existence of finite non-tangential limits for every function in $g K_I$ at $\zeta$ is that $g$ has a finite non-tangential limit at $\zeta$ (since indeed $g$ is extremal and so $g \in g K_{I}$). 
As we will see with the next example, it is not possible to deduce the existence of the non-tangential
limits for every function in $gK_I$ merely from the uniform norm boundedness of the reproducing
kernels in a Stolz angles.

\begin{example}
Given any inner function $I$, with $I(0) = 0$, we have already cited Sarason's result \cite{Sa88}
stating that every isometric multiplier on $K_I$ is of the form 
\bea\label{defg}
 g=\frac{a}{1-Ib},
\eea
where $a$ and $b$ are functions in the ball of $\Hi$ with $|a|^2+|b|^2=1$ a.e.\ on $\T$.
Let $\Lambda_1=\{1-4^{-n}\}_{n \geq 1}$ and 
$\Lambda_2=\{1-2^{-n}\}_{n \geq 1}$, which are both interpolating 
sequences. Observe that $\Lambda_2$ consists
of the sequence $\Lambda_1$ to which we have added in a certain way the 
(pseudohyperbolic) midpoints
of two consecutive points of $\Lambda_1$. Let $B_i$ be the Blaschke
product whose zeros are $\Lambda_i$. In particular, $B_1$ vanishes on $\Lambda_1$ and is big on
$\Lambda_2\setminus \Lambda_1$, say $|B_1(\lambda)|\ge\delta>0$ for
$\lambda\in\Lambda_2\setminus \Lambda_1$ (this comes from the fact that $B_1$
is an interpolating Blaschke product, and that on $\Lambda_2\setminus \Lambda_1$
we are pseudohyperbolically far from the zeros of $B_1$).

Next, define
\[
 a=\frac{1}{2}+\frac{1}{4}B_1
\]
which is a function oscillating on $\Lambda_2$ between the values $1/2$ (on
$\Lambda_1$) and $|a(\lambda)-1/2|\ge \delta/4$, 
$\lambda\in\Lambda_2\setminus \Lambda_1$.

By construction we have $1/4\le |a| \le 3/4$. In particular
there is an outer function $b_0$ in the ball of $\Hi$ such that $|a|^2+|b_0|^2=1$ a.e.\ on $\T$.
Moreover $0<\sqrt{7}/4\le |b_0|=\sqrt{1-|a|^2}\le \sqrt{15}/4<1$ a.e.\ on $\T$ 
and this extends by the
maximum/minimum principles to the disk ($b_0$ being outer).
As a consequence, for
every inner function $J$, the function $g=a/(1-Jb_0)$ will be a bounded outer function
(actually invertible).
We will use this fact in particular for $J=IB_2$: 
\[
 g=\frac{a}{1-Jb_0}=\frac{a}{1-I\underbrace{B_2b_0}_{=b \text{ in \eqref{defg}}}}.
\]
Then
\[
 \left|g(\lambda)-\frac{1}{2}\right|=\left|a(\lambda)-\frac{1}{2}\right|
 \left\{\begin{array}{ll}
  =0 & \mbox{when }\lambda\in\Lambda_1\\
  \ge \frac{\dst \delta}{\dst 4} &\mbox{when }\lambda\in\Lambda_2\setminus\Lambda_1.
 \end{array}
 \right.
\]
Hence $g$ has no limit at $\zeta=1$. 
Choose now an inner function $I$ such that $\zeta=1\in \ADC(I)$. Then, the kernels $k_{\lambda}^M=\overline{g(\lambda)}gk_{\lambda}^I$ are
uniformly bounded in a fixed Stolz region $\Gamma_{\alpha}(1)$ at $\zeta=1$. Indeed $g$ is bounded,
the kernels $k_{\lambda}^I$ are uniformly bounded in $\Gamma_{\alpha}(1)$ by $1\in \ADC(I)$
(see Theorem \ref{thmACbis}),
and $\|k_{\lambda}^M\|=|g(\lambda)|\|k_{\lambda}^I\|$.
\end{example}


\begin{theorem} \label{MNTL}
Let $M=gK_I$ be a non-trivial nearly invariant subspace with extremal function $g$ and
corresponding inner function $I$. Then every function in $M$ has a finite non-tangential limit at $\zeta \in \T$ if and only if the following two conditions are satisfied:
\begin{enumerate}
\item $g$ has a finite non-tangential limit at $\zeta$. 
\item For every Stolz region $\Gamma_{\alpha}(\zeta)$ at $\zeta$, the family 
$F^{M}_{\alpha, \zeta}: =\{k_{\lambda}^M:\lambda\in\Gamma_{\alpha}(\zeta)\}$ is uniformly norm bounded, i.e., 
$$\sup\left\{\|k_{\lambda}^{M}\|^2 = |g(\lambda)|^2 \frac{1 -
    I(\lambda)|^2}{1 - |\lambda|^2}: \lambda \in
  \Gamma_{\alpha}(\zeta)\right\} 
< \infty.$$
\end{enumerate}
\end{theorem}

\begin{proof}
If every function in $M$ has a finite non-tangential limit at $\zeta$, then (1) holds since $g \in M$. 
Condition (2) holds by the uniform boundedness principle. 

Now suppose that conditions (1) and (2) hold. 

\underline{Case 1:} $\zeta\in \ADC(I)$. In this case, every function in $M$ has a non-tangential
boundary limit at $\zeta$ since every function in $K_I$ has as well as $g$ does.

\underline{Case 2:} $\zeta\not\in\ADC(I)$.
From the Ahern-Clark theorem (Theorem \ref{thmACbis}), we know that
$\zeta \in \ADC(I)$ if and only if
\[
 \|k_{\lambda}^I\|^2=\frac{1-|I(\lambda)|^2}{1-|\lambda|^2}
\]
is uniformly bounded in Stolz regions $\Gamma_{\alpha}(\zeta)$.
Hence, if $\zeta\not\in\ADC(I)$, then there is a sequence $\Lambda:=\{\lambda_n\}_{n \geq 1}$
in $\Gamma_{\alpha}(\zeta)$ such that $ \|k_{\lambda_n}^I\|^2\to +\infty$.
Recall that $k_{\lambda}^M=\overline{g(\lambda)}gk_{\lambda}^I$ so that
$\|k_{\lambda_n}^M\|=|g(\lambda_n)|\|k_{\lambda_n}^I\|$. In order that this
last expression is uniformly bounded it is necessary that $g(\lambda_n)\to 0$
when $n\to\infty$.

Now, since $F^{M}_{\alpha, \zeta}$ is a uniformly bounded
family, there exists a subsequence $\Lambda':= \{\lambda_{n_j}\}_{j \geq 1}$ so that $k^M_{\lambda_{n_j}}$ converges weakly
to a function $k_{\zeta}^{M,\Lambda'}$.

Let us introduce the following operator
\[
  Q_N=U_g A_{\overline{z}}^N U_g^*
\]
from $M=gK_I$ to itself.
Let $f=gh\in gK_I=M$. Recall that $U_g^*f=h$. Then
\beqa
 (Q_Nf)(z)&=&\Big((U_g A_{\overline{z}}^N U_g^*)f\Big)(z)=\Big((U_g A_{\overline{z}}^N h\Big)(z)=
 g(z)\frac{h(z)-\sum_{m=0}^{N-1}\widehat{h}(m) z^m}{z^N}\\
 &=&\frac{f(z)}{z^N}-\frac{g(z)}{z^N}
 \sum_{m=0}^{N-1} \widehat{h}(m) z^m.
\eeqa
Now since $Q_Nf\in M$, we get, from the weak convergence of $k_{\lambda_{n_{j}}}^{M} \to k_{\zeta}^{M, \Lambda'}$ as $j \to \infty$, 
\begin{equation} \label{QNf}
  \langle Q_Nf, k_{\zeta}^{M,\Lambda'}\rangle
 =\lim_{j\to\infty}  (Q_Nf)(\lambda_{n_j})
 =\lim_{j\to\infty}\left(\frac{f(\lambda_{n_j})}{\lambda_{n_j}^N}-\frac{g(\lambda_{n_j})}
  {\lambda_{n_j}^N}
 \sum_{m=0}^{N-1}\widehat{h}(m)\lambda_{n_j}^m \right).
\end{equation}
The last sum appearing above is a harmless polynomial converging to some 
number when $j\to\infty$. Also $\lambda_{n_j}\to \zeta$
and $g(\lambda_{n_j})\to 0$ when $j\to\infty$. Hence
the limit in \eqref{QNf} is equal to
\[
  \langle Q_Nf, k_{\zeta}^{M,\Lambda'}\rangle
  =\frac{f_{\zeta}^{\Lambda'}}{\zeta^N},
\]
where 
\[
 f_{\zeta}^{\Lambda'}=\langle f, k_{\zeta}^{M,\Lambda'}\rangle,
\]
which exists by construction.

Since $$\|Q_{N} f\| = \|U_{g} A_{\overline{z}}^{N} h\| = \|g A_{\overline{z}}^{N} h\|  = \|A_{\overline{z}}^{N} h\|  \to 0 \quad \mbox{as $N \to \infty$},$$
 we have
\[
  \lim_{N\to\infty}\frac{f_{\zeta}^{\Lambda'}}{\zeta^N}
 =\lim_{N\to\infty} \langle Q_Nf, k_{\zeta}^{M,\Lambda'}\rangle
 =0
\]
which is possible precisely when $f_{\zeta}^{\Lambda'}=0$.

The above construction is independent of the choice of the sequence $\{\lambda_n\}_{n \geq 1}$
so that for every such sequence there is a subsequence $\{\lambda_{n_j}\}_{j \geq 1}$ such
that $f(\lambda_{n_j})\to 0$ when $j\to\infty$. As in the proof of Theorem
\ref{thmACbis} we conclude that $\angle\lim_{\lambda\to \zeta}f(\lambda)$ exists
and is, in fact, equal to 0.

\end{proof}

\section{The dichotomy}

For a nearly invariant subspace 
$M = g K_{I}$, let 
$$N^{M} = \left\{\zeta \in \T: \angle \lim_{\lambda \to \zeta} f(\lambda) = 0 \;  \forall f \in M\right\}.$$
The sets $\ADC(I)$ and $N^M$ are not necessarily 
disjoint. If the function $g$ vanishes at a point $\zeta\in \ADC(I)$ then
automatically every function $f\in gK_I$ vanishes at $\zeta$ so that $\zeta\in N^M$.

Another situation is the following.
Observing that every function $f\in H^2$ satisfies the following 
well known growth condition
\bea\label{growth}
 |f(z)|=O\left(\frac{1}{\sqrt{1-|z|}}\right),
\eea
we see that whenever $|g(z)|=o(\sqrt{|z-\zeta|})$ in a Stolz domain $\Gamma_{\alpha}(\zeta)$
based on $\zeta$, 
 every function $gh$ tends non tangentially to $0$ at $\zeta$, and so $\zeta\in N^M$
independent of $I$. Notice that for growth in a Stolz domain $\Gamma_{\alpha}(\zeta)$ we can replace
``big-Oh'' by ``little-oh'' in \eqref{growth}, and then ``little-oh'' by ``big-oh'' in
the growth of $g$ in $\Gamma_{\alpha}(\zeta)$. The main theorem of this section is the following. 


\begin{theorem}\label{thmdicho}
Let $M$ be a non-trivial nearly invariant subspace with extremal function $g$ and
associated inner function $I$.
Suppose that every function in $gK_I$ has a finite non-tangential limit at $\zeta\in\T$.
Then $\zeta\in \ADC(I)$ or $\zeta\in N^M$.
\end{theorem}

Note again that the intersection $\ADC(I)\cap N^M$ can be non-empty.

\begin{proof}[Proof of Theorem \ref{thmdicho}]
If every function in $gK_I$ has a non-tangential limit at $\zeta\in \T$, then,
in particular, $g$ has a non-tangential limit at $\zeta$. As a reminder, recall that $g$ belongs to $g K_I$ since $I(0) = 0$ (and thus the constants belong to $K_I$).

If 
$$\angle \lim_{\lambda \to \zeta} g(\lambda) \not = 0,$$ then, since every function
$f=gh\in gK_I$ has a finite non-tangential limit at $\zeta$, we can divide by $g$ to
get $h=f/g\in K_I$ which has a non-tangential limit at $\zeta$. So, every function
in $K_I$ has non-tangential limit and we conclude from \cite{AC70} that $\zeta \in \ADC(I)$. 

Now suppose that
\bea\label{limg0}
 \angle \lim_{\lambda \to \zeta} g(\lambda) = 0.
\eea
First note that if every function in $gK_I$ has a non-tangential limit 
at $\zeta\in\T$, then, as already mentioned, there exists a reproducing
kernel $k_{\zeta}^M\in M$ at $\zeta$. Let now $\{\lambda_n\}_{n\ge 1}$
be any sequence with $\lambda_n\in \Gamma_{\alpha}(\zeta)$ tending to
$\zeta$. Then for $z\in\D$,
\[
 k_{\zeta}^M(z)=\langle k_{\zeta}^M,k_z^M\rangle=
 \overline{\langle k_z^M,k_{\zeta}^M\rangle}=\overline{k_z^M(\zeta)}
 =\lim_{n\to\infty}\overline{k_z^M(\lambda_n)}
 =\lim_{n\to\infty}\left( g(z)\overline{g(\lambda_n)}
  \frac{1-\overline{I(\lambda_n)}I(z)}{1-\overline{\lambda_n}z}\right)
\]
Since $z \in \D$ is fixed, by \eqref{limg0}, the above limit will be zero, which forces the kernel $k_{\zeta}^M$
to vanish identically. Thus for every $f\in M=gK_I$ we have
\[
 f(\zeta)=\lim_{n\to 0}\langle f,k^M_{\lambda_n}\rangle
 =\langle f,0\rangle=0.
\]
\end{proof}

As a consequence of this result, if $\zeta\in N^M\setminus \ADC(I)$,
then the point evaluation operator at $\zeta$ is just the zero-functional.
We state this observation as a separate result.

\begin{corollary}
Let $M$ be a non-trivial nearly invariant subspace with extremal function $g$ and
associated inner function $I$.
On $M$, the only non-zero point evaluation functionals at a point $\zeta\in\T$
are those for which $\zeta\in\ADC(I)$ and $g$ admits a non-tangential
limit different from zero at $\zeta$. 
\end{corollary}

\bibliography{gk2i}

\end{document}